\documentclass[11pt]{amsart}
\usepackage{amsmath,amssymb,amsthm}
\usepackage{enumitem}
\usepackage{xcolor}
\usepackage[colorlinks=true,linkcolor=blue,citecolor=blue,urlcolor=blue]{hyperref}
\usepackage{microtype}

\newtheorem{theorem}{Theorem}[section]
\newtheorem{lemma}[theorem]{Lemma}
\newtheorem{proposition}[theorem]{Proposition}
\newtheorem{corollary}[theorem]{Corollary}
\newtheorem{conjecture}[theorem]{Conjecture}

\newtheorem{sublemma}{}[theorem]
\theoremstyle{definition}

\DeclareMathOperator{\cl}{cl}
\DeclareMathOperator{\si}{si}
\DeclareMathOperator{\GF}{GF}
\DeclareMathOperator{\PG}{PG}
\DeclareMathOperator{\AG}{AG}

\newcommand{\cB}{\mathcal B}
\newcommand{\cF}{\mathcal F}

\setlist[enumerate,1]{label=\textup{(\roman*)},leftmargin=2.2em,itemsep=2pt,topsep=4pt}
\title[Unavoidable flats and uniform minors]{Unavoidable flats in matroids excluding a uniform minor}
\author{Wayne Ge}
\address{Auburn University, Auburn, AL}
\email{yig0001@auburn.edu}
\date{October 27, 2026}

\begin{document}
\begin{abstract}
We extend a theorem of Geelen, Kroeker, and Nelson on unavoidable flats in $\GF(q)$-representable matroids. For all integers $1\le a\le b$ and $n\ge1$, we prove that every simple matroid of sufficiently large rank with no $U_{a,b}$-minor has a rank-$n$ flat whose restriction is free or isomorphic to a projective or affine geometry over a finite field.
\end{abstract}
\maketitle

\section{Introduction}\label{sec:introduction}

Geelen, Kroeker, and Nelson~\cite[Main Theorem]{GKN26} proved the following theorem on unavoidable flats in $\GF(q)$-representable matroids.

\begin{theorem}\label{thm:finite-field}
For each prime power $q$ and positive integer $n$, there is an integer $N_{\ref{thm:finite-field}}(q,n)$ such that every simple $\GF(q)$-representable matroid of rank at least $N_{\ref{thm:finite-field}}(q,n)$ has a rank-$n$ flat whose restriction is free or isomorphic to $\PG(n-1,q_0)$ or $\AG(n-1,q_0)$ for a subfield $\GF(q_0)$ of $\GF(q)$.
\end{theorem}

Theorem~\ref{thm:finite-field} first appeared in Kroeker's dissertation~\cite[pp.~32--33]{Kroeker25} and extends an earlier theorem of Geelen and Kroeker~\cite[Theorem~1.2]{GK24} for matroids representable over a fixed prime field. These results advance some Ramsey-type conjectures discussed by Jim Geelen in an online talk~\cite{GeelenRamsey21} and an earlier post~\cite{Geelen21}. We find the following ``wild version'' particularly appealing.

\begin{conjecture}\label{conj:wild-version}
For all integers $\ell,n\ge2$, there is an integer $N_{\ref{conj:wild-version}}(\ell,n)$ such that every simple matroid of rank at least $N_{\ref{conj:wild-version}}(\ell,n)$ with no line of length at least $\ell+1$ has a rank-$n$ flat whose restriction is free or isomorphic to a projective or affine geometry over a finite field.
\end{conjecture}

The line-length hypothesis in Conjecture~\ref{conj:wild-version} bounds the field orders of the projective and affine geometries. Lines in $\PG(n-1,q)$ and $\AG(n-1,q)$ have $q+1$ and $q$ elements, respectively, so the projective and affine geometries have field orders at most $\ell-1$ and $\ell$, respectively.

Representability plays an essential role in the proofs of~\cite{GK24} and~\cite{GKN26}. We replace this hypothesis by the exclusion of a fixed uniform minor. For each prime power $q$, every $\GF(q)$-representable matroid has neither a $U_{2,q+2}$-minor nor a $U_{q,q+2}$-minor. On the other hand, matroids excluding a fixed uniform minor, in general, need not be representable. We prove the following theorem.

\begin{theorem}\label{thm:unavoidable}\label{thm:main}
For all integers $1\le a\le b$ and $n\ge1$, there is an integer $R_{\ref{thm:unavoidable}}(a,b,n)$ such that every simple matroid $M$ with no $U_{a,b}$-minor and $r(M)\ge R_{\ref{thm:unavoidable}}(a,b,n)$ has a rank-$n$ flat $F$ such that $M|F$ is either free or isomorphic to a projective or affine geometry over a finite field.
\end{theorem}

The case $a=b$ is vacuous at sufficiently large rank, since excluding $U_{a,a}$ forces the rank to be less than $a$. For $a<b$, let $s=\max\{2,a,b-a\}$. Contracting $s-a$ elements of $U_{s,2s}$ and deleting $s+a-b$ further elements gives $U_{a,b}$. Thus every matroid with no $U_{a,b}$-minor also has no $U_{s,2s}$-minor. Consequently, Theorem~\ref{thm:unavoidable} follows immediately from the following.

\begin{theorem}\label{thm:balanced-uniform-free}
For all integers $s\ge2$ and $n\ge1$, there is an integer $R_{\ref{thm:balanced-uniform-free}}(s,n)$ such that every simple matroid $M$ with no $U_{s,2s}$-minor and $r(M)\ge R_{\ref{thm:balanced-uniform-free}}(s,n)$ has a rank-$n$ flat $F$ such that $M|F$ is either free or isomorphic to a projective or affine geometry over a finite field.
\end{theorem}

We use the following representability theorem to prove Theorem~\ref{thm:balanced-uniform-free} by applying Theorem~\ref{thm:finite-field} inside the resulting flat. For an integer $s\ge2$, let $\cB(s)$ be the class of matroids with no $U_{s,2s}$-minor.

\begin{theorem}\label{thm:representable-flat}\label{thm:rep-flat}\label{thm:bounded-representable-flat}
For all integers $s,t\ge2$, there is an integer $R_{\ref{thm:representable-flat}}(s,t)$ such that every matroid $M\in\cB(s)$ with $r(M)\ge R_{\ref{thm:representable-flat}}(s,t)$ has a rank-$t$ flat $H$ for which $M|H$ is $\GF(q)$-representable for some prime power $q<2s$.
\end{theorem}

We note that the bound on $q$ allows us to take a finite maximum of the rank thresholds in Theorem~\ref{thm:finite-field}.

Nelson and Norin~\cite[Theorem~1]{NN21} proved that excluding a large independent flat forces an exponential lower bound on the number of points in terms of rank. We adapt their argument to obtain an exponential lower bound on the number of sets of bounded rank needed to cover the ground set. Combining this bound with a result of Nelson~\cite[Theorem~1.1]{N15} gives a large projective geometry minor in every sufficiently high-rank simple matroid in $\cB(s)$ with no large independent flat.

The main difficulty in proving Theorem~\ref{thm:representable-flat} is passing from a projective geometry minor to a flat of the original matroid. Lemma~\ref{lem:representable-flat} is the central step of our argument and may be of independent interest. For fixed integers $s,t\ge2$, it shows that a simple matroid in $\cB(s)$ with a sufficiently large projective geometry minor has a rank-$t$ flat representable over a field with fewer than $2s$ elements.

We give the required tools and the exponential covering argument in Section~\ref{sec:inputs}, and prove Lemma~\ref{lem:representable-flat} in Section~\ref{sec:lifting}. Section~\ref{sec:main-proof} contains the proofs of Theorems~\ref{thm:representable-flat} and~\ref{thm:balanced-uniform-free}. Section~\ref{sec:consequences} presents some Ramsey and structural consequences.

\section{Preliminaries}\label{sec:inputs}\label{sec:tools}

We follow Oxley~\cite{Oxley} for matroid terminology. All matroids considered in this paper are finite and may have loops or parallel elements unless simplicity is explicitly assumed. We use $\PG(n-1,q)$ and $\AG(n-1,q)$ for the projective and affine geometries of rank $n$ over $\GF(q)$. We write $\si(M)$ for a simplification of $M$. A \emph{contraction-minor} of $M$ is a matroid $M/C$ for some $C\subseteq E(M)$. An \emph{independent flat} is a flat that is an independent set. Two sets $X,Y\subseteq E(M)$ are \emph{skew} if $r_M(X\cup Y)=r_M(X)+r_M(Y)$. For an integer $a\ge1$, the \emph{$a$-covering number} $\tau_a(M)$ is the minimum number of sets of rank at most $a$ whose union is $E(M)$.

We use the following theorem of Nelson~\cite[Theorem~1.1]{N15}, which refines an earlier result of Geelen and Nelson~\cite[Theorem~1.1]{GN15}.

\begin{theorem}\label{thm:covering-growth}
Let $a\ge1$ be an integer, and let $\mathcal M$ be a minor-closed class of matroids. There is an integer $c\ge1$ such that at least one of the following holds.
\begin{enumerate}
\item $\tau_a(M)\le r(M)^c$ for every $M\in\mathcal M$ with $r(M)\ge2$.
\item There is a prime power $q$ such that $\mathcal M$ contains every $\GF(q)$-representable matroid and $\tau_a(M)\le cq^{r(M)}$ for every $M\in\mathcal M$ with $r(M)\ge2$.
\item $\mathcal M$ contains every rank-$(a+1)$ uniform matroid.
\end{enumerate}
\end{theorem}

To pass from a projective geometry minor to a representable flat, we also use the following result of Geelen and Nelson~\cite[Lemma~7.2]{GN17}, which we refer to as the local representability lemma.

\begin{theorem}\label{thm:local-rep}
There is a function $f_{\ref{thm:local-rep}}$ from triples of integers at least two to nonnegative integers with the following property. Let $s,h,t\ge2$, let $q$ be a prime power, and let $N$ have a spanning $\PG(r(N)-1,q)$-restriction. If $N$ has no $U_{s,2s}$-minor and no rank-$h$ projective geometry minor over a finite field of order greater than $q$, then there is a set $C\subseteq E(N)$ with $|C|\le f_{\ref{thm:local-rep}}(s,h,t)$ such that every rank-$t$ restriction of $N/C$ is $\GF(q)$-representable.
\end{theorem}

An equivalent form of the following fact is stated in~\cite[p.~79]{GN17}. It follows from the fact that $U_{s,q}$ is a restriction of $\PG(s-1,q)$ when $2\le s\le q$.

\begin{proposition}\label{prop:field-bound}\label{lem:field-bound}
Let $s\ge2$ be an integer, and let $M$ be a matroid with no $U_{s,2s}$-minor. If $M$ has a minor $N\cong\PG(d-1,q)$, where $d\ge s$ and $q$ is a prime power, then $q<2s$.
\end{proposition}

We also use the following elementary fact about skew sets.

\begin{proposition}\label{prop:skew-restriction}\label{lem:skew-lift}
Let $M$ be a matroid, and let $A,B\subseteq E(M)$ be disjoint skew sets. Then $M|A=(M/B)|A$.
\end{proposition}

\subsection{Independent flats and covering numbers}\label{sec:projective-minors}\label{sec:covering-density}

Nelson and Norin~\cite[Lemma~4]{NN21} observed that independent flats in simplifications of contraction-minors lift to independent flats of the original simple matroid. We include the proof.

\begin{lemma}\label{lem:independent-flat-contraction}
Let $M$ be simple, let $C\subseteq E(M)$, and let $N=\si(M/C)$. Every independent flat of $N$ is an independent flat of $M$.
\end{lemma}
\begin{proof}
Let $I$ be an independent flat of $N$, and let $H=\cl_M(I)$. Since $I$ is independent in $M/C$, it is independent in $M$ and skew to $C$. Thus $H$ is skew to $C$. As $M$ is simple, $H\cap C=\emptyset$, and Proposition~\ref{prop:skew-restriction} gives $M|H=(M/C)|H$. Every element of $H$ is therefore a nonloop of $M/C$ spanned by $I$. Since $I$ is a flat of $N$, every element of $H-I$ is parallel in $M/C$ to an element of $I$. Since $(M/C)|H=M|H$ is simple, $H-I=\emptyset$. Thus $H=I$, proving the lemma.
\end{proof}

The following lemma adapts the contraction argument used by Nelson and Norin in the proof of~\cite[Theorem~5]{NN21}.

\begin{lemma}\label{lem:exponential-covering}
Let $k\ge2$, and let $M$ be a simple matroid of positive rank $r$ with no rank-$k$ independent flat. For every integer $a\ge1$,
\[
\tau_a(M)\ge 2^{\lceil r/(k-1)\rceil-a}.
\]
\end{lemma}
\begin{proof}
Let $b=k-1$. We first prove the following.

\begin{sublemma}\label{claim:covering-set}
There is a set $X\subseteq E(M)$ such that $|X|\ge2^{\lceil r/b\rceil-1}$ and $|A\cap X|\le2^{r_M(A)-1}$ for every nonempty set $A\subseteq E(M)$.
\end{sublemma}

We proceed by induction on $r$. If $1\le r\le b$, let $X$ consist of one element. Then $|X|=2^{\lceil r/b\rceil-1}=1$, and the intersection bound holds because $r_M(A)\ge1$.

Suppose that $r>b$. Choose an independent flat $S$ of maximum size, and let $t=|S|$. Clearly, $t\le b$, otherwise $S$ contains a rank-$k$ independent flat of $M$. Since $M$ is simple, each singleton is an independent flat, so $1\le t\le b<r$. Let $N=\si(M/S)$. Lemma~\ref{lem:independent-flat-contraction} implies that $N$ has no rank-$k$ independent flat, and $r(N)=r-t\ge1$.

Since $S$ is a flat, $M/S$ is loopless. If $\{e\}$ were a singleton parallel class, then $\cl_M(S\cup\{e\})=S\cup\{e\}$, so $S\cup\{e\}$ would be a larger independent flat, contrary to the choice of $S$. Thus every parallel class of $M/S$ has at least two elements.

By induction, there is a set $Y\subseteq E(N)$ such that $|Y|\ge2^{\lceil(r-t)/b\rceil-1}$ and $|B\cap Y|\le2^{r_N(B)-1}$ for every nonempty set $B\subseteq E(N)$. For each $y\in Y$, choose $y'\in E(M)-(S\cup\{y\})$ parallel to $y$ in $M/S$, and let $X=\bigcup_{y\in Y}\{y,y'\}$. Then
\[
|X|=2|Y|\ge2^{\lceil(r-t)/b\rceil}\ge2^{\lceil r/b\rceil-1},
\]
where the last inequality uses $t\le b$.

Let $A\subseteq E(M)$ be nonempty, and let $P$ be the set of elements of $N$ whose parallel class in $M/S$ meets $A-S$. Then $r_N(P)=r_M(A\cup S)-t$. If $P=\emptyset$, then $A\subseteq S$, so $A\cap X=\emptyset$.

Suppose that $P\ne\emptyset$. If $A$ is skew to $S$, then $r_N(P)=r_M(A)$. By Proposition~\ref{prop:skew-restriction}, $(M/S)|A=M|A$ is simple, so $A$ meets each parallel class of $M/S$ in at most one element. Consequently,
\[
|A\cap X|\le |P\cap Y|\le2^{r_N(P)-1}=2^{r_M(A)-1}.
\]
If $A$ is not skew to $S$, then $r_N(P)\le r_M(A)-1$. Each parallel class contributes at most two elements to $X$, so
\[
|A\cap X|\le2|P\cap Y|\le2^{r_N(P)}\le2^{r_M(A)-1}.
\]
Hence~\ref{claim:covering-set} holds.

Let $F_1,\ldots,F_m$ cover $E(M)$, with $r_M(F_i)\le a$ for every $i$. Thus~\ref{claim:covering-set} gives
\[
2^{\lceil r/b\rceil-1}\le |X|\le\sum_{i=1}^m|F_i\cap X|\le m2^{a-1}.
\]
It follows that $m\ge2^{\lceil r/b\rceil-a}$. This proves the lemma.
\end{proof}

Combining Lemma~\ref{lem:exponential-covering} with Theorem~\ref{thm:covering-growth} gives the required projective geometry minor.

\begin{corollary}\label{cor:pg-minor}
For all integers $s,k,n\ge2$, there is an integer $R_{\ref{cor:pg-minor}}(s,k,n)$ such that every simple matroid $M\in\cB(s)$ with no rank-$k$ independent flat and $r(M)\ge R_{\ref{cor:pg-minor}}(s,k,n)$ has a rank-$n$ projective geometry minor.
\end{corollary}
\begin{proof}
Fix $s,k,n\ge2$, and let $\mathcal D$ be the class of matroids with no $U_{s,2s}$-minor and no rank-$n$ projective geometry minor. Apply Theorem~\ref{thm:covering-growth} with $a=s-1$. The class $\mathcal D$ does not contain every rank-$s$ uniform matroid, since it excludes $U_{s,2s}$. For every prime power $q$, it excludes the $\GF(q)$-representable matroid $\PG(n-1,q)$, so it does not contain all $\GF(q)$-representable matroids. Therefore there is an integer $c\ge1$, depending only on $s,n$, such that $\tau_{s-1}(N)\le r(N)^c$ for every $N\in\mathcal D$ with $r(N)\ge2$.

Choose an integer $R=R_{\ref{cor:pg-minor}}(s,k,n)\ge\max\{2,k,n\}$ such that
\[
2^{r/(k-1)-(s-1)}>r^c\qquad\text{for every integer }r\ge R.
\]
Let $M$ be a simple matroid in $\cB(s)$ of rank $r\ge R$ with no rank-$k$ independent flat. By Lemma~\ref{lem:exponential-covering},
\[
\tau_{s-1}(M)\ge2^{\lceil r/(k-1)\rceil-(s-1)}\ge2^{r/(k-1)-(s-1)}>r^c.
\]
Thus $M\notin\mathcal D$. Since $M\in\cB(s)$, it has a rank-$n$ projective geometry minor.
\end{proof}

\section{Representable flats from projective geometry minors}\label{sec:lifting}

The following lemma is the key step in our proof and may be of independent interest. It shows that a sufficiently large projective geometry minor in a simple matroid in $\cB(s)$ forces a large representable flat. The argument makes a maximal choice among finitely many field orders, as in Geelen and Nelson~\cite[Lemma~8.1]{GN17}.

\begin{lemma}\label{lem:representable-flat}\label{lem:pg-to-flat}
For all integers $s,t\ge2$, there exists an integer $u_{\ref{lem:representable-flat}}(s,t)$ such that every simple matroid $M\in\cB(s)$ with a rank-$u_{\ref{lem:representable-flat}}(s,t)$ projective geometry minor has a rank-$t$ flat $H$ for which $M|H$ is $\GF(q)$-representable for some prime power $q<2s$.
\end{lemma}
\begin{proof}
Fix $s,t\ge2$, and let $\mathcal C$ be the class of simple matroids in $\cB(s)$ with no rank-$t$ flat representable over a field of order less than $2s$. Suppose, for a contradiction, that for every positive integer $r$, some member of $\mathcal C$ has a rank-$r$ projective geometry minor. By Proposition~\ref{prop:field-bound}, each such minor of rank at least $s$ is over a field of order less than $2s$.

Let $Q_s$ be the finite set of prime powers less than $2s$. Since $Q_s$ is finite, there is a largest $q\in Q_s$ such that, for every positive integer $r$, some member of $\mathcal C$ has a $\PG(r-1,q)$-minor. By maximality of $q$, for each $q'\in Q_s$ with $q'>q$, there is a positive integer $c(q')$ such that no member of $\mathcal C$ has a $\PG(c(q')-1,q')$-minor. Let
\[
c=\max\bigl(\{s\}\cup\{c(q'):q'\in Q_s,\ q'>q\}\bigr).
\]
Then no member of $\mathcal C$ has a rank-$c$ projective geometry minor over a finite field of order greater than $q$.

Let $d=f_{\ref{thm:local-rep}}(s,c,t)$. Choose $M\in\mathcal C$ with a $\PG(t+d-1,q)$-minor. By the Scum Theorem~\cite[Theorem~3.3.1]{Oxley}, there is a contraction-minor $N=M/C_0$ of rank $t+d$ with a spanning $\PG(t+d-1,q)$-restriction.

Since $N$ is a minor of $M$, it has no $U_{s,2s}$-minor and no rank-$c$ projective geometry minor over a finite field of order greater than $q$. By Theorem~\ref{thm:local-rep}, there is a set $C_1\subseteq E(N)$ with $|C_1|\le d$ such that every rank-$t$ restriction of $N/C_1$ is $\GF(q)$-representable. Moreover, $r(N/C_1)\ge r(N)-|C_1|\ge t$.

Let $B$ be a $t$-element independent set of $N/C_1$. Since $N/C_1=M/(C_0\cup C_1)$, the set $B$ is independent in $M$ and skew to $C_0\cup C_1$. Taking closure preserves skewness, so $\cl_M(B)$ is a rank-$t$ flat of $M$ skew to $C_0\cup C_1$. As $M$ is simple, $\cl_M(B)\cap(C_0\cup C_1)=\emptyset$. Proposition~\ref{prop:skew-restriction} gives
\[
M|\cl_M(B)=\bigl(M/(C_0\cup C_1)\bigr)|\cl_M(B)=(N/C_1)|\cl_M(B).
\]
The last matroid is a rank-$t$ restriction of $N/C_1$ and is therefore $\GF(q)$-representable. Hence $\cl_M(B)$ is a rank-$t$ $\GF(q)$-representable flat of $M$. Since $q<2s$, this contradicts $M\in\mathcal C$.

Consequently, there is a positive integer $u_{\ref{lem:representable-flat}}(s,t)$ such that no member of $\mathcal C$ has a projective geometry minor of rank $u_{\ref{lem:representable-flat}}(s,t)$. This proves the lemma.
\end{proof}

\section{The proofs of the unavoidable-flats theorems}\label{sec:main-proof}\label{sec:proof}

We are now ready to prove the main results of this paper.

\begin{proof}[Proof of Theorem~\ref{thm:representable-flat}]
Fix $s,t\ge2$, let $u=\max\{s,u_{\ref{lem:representable-flat}}(s,t)\}$, and let $R_{\ref{thm:representable-flat}}(s,t)=R_{\ref{cor:pg-minor}}(s,t,u)$. Let $M\in\cB(s)$ have rank at least $R_{\ref{thm:representable-flat}}(s,t)$.

Suppose first that $M$ is simple. If $M$ has a rank-$t$ independent flat, then its restriction is $\GF(2)$-representable and $2<2s$. Otherwise Corollary~\ref{cor:pg-minor} gives a rank-$u$ projective geometry minor. Restricting this geometry to rank $u_{\ref{lem:representable-flat}}(s,t)$ and applying Lemma~\ref{lem:representable-flat} gives the desired flat.

For a general matroid $M$, let $S=\si(M)$. Applying the simple case gives a rank-$t$ flat $H_0$ of $S$ such that $S|H_0$ is $\GF(q)$-representable for some $q<2s$. Let $H=\cl_M(H_0)$. Then $r_M(H)=t$, and $M|H$ is obtained from $S|H_0$ by adding loops and parallel elements, so it is also $\GF(q)$-representable.
\end{proof}

We prove a more detailed version of Theorem~\ref{thm:balanced-uniform-free}.

\begin{theorem}\label{thm:detailed}
For all integers $s\ge2$ and $n\ge1$, there is an integer $R_{\ref{thm:detailed}}(s,n)$ such that every simple matroid $M\in\cB(s)$ with $r(M)\ge R_{\ref{thm:detailed}}(s,n)$ has a rank-$n$ flat $F$ for which $M|F$ is isomorphic to $U_{n,n}$ or to a projective or affine geometry of rank $n$ over a field with fewer than $2s$ elements.
\end{theorem}
\begin{proof}
Fix $s\ge2$ and $n\ge1$, let $Q_s$ be the set of prime powers less than $2s$, and let
\[
t(s,n)=\max\bigl(\{2,n\}\cup\{N_{\ref{thm:finite-field}}(q,n):q\in Q_s\}\bigr).
\]
Let $t=t(s,n)$, let $R_{\ref{thm:detailed}}(s,n)=R_{\ref{thm:representable-flat}}(s,t)$, and let $M\in\cB(s)$ be simple with rank at least $R_{\ref{thm:detailed}}(s,n)$.

By Theorem~\ref{thm:representable-flat}, there is a rank-$t$ flat $H$ of $M$ such that $M|H$ is $\GF(q)$-representable for some $q\in Q_s$. Since $M|H$ is simple and $t\ge N_{\ref{thm:finite-field}}(q,n)$, Theorem~\ref{thm:finite-field} gives a rank-$n$ flat $F$ of $M|H$ whose restriction is free or isomorphic to $\PG(n-1,q_0)$ or $\AG(n-1,q_0)$ for a subfield $\GF(q_0)$ of $\GF(q)$.

Since $H$ is a flat of $M$, so is $F$. If $M|F$ is a projective or affine geometry over $\GF(q_0)$, then $q_0\le q<2s$. This proves the theorem.
\end{proof}

\section{Ramsey and structural consequences}\label{sec:consequences}

\subsection{Colorings of flats}

The three families in Theorem~\ref{thm:unavoidable} support Ramsey theorems for colorings of flats. Geelen, Kroeker, and Nelson~\cite[Corollary~1.1]{GKN26} combined Ramsey's Theorem~\cite{Ramsey} and the geometric Ramsey theorems of Graham, Leeb, and Rothschild~\cite{GLR} to obtain such a theorem for matroids representable over a fixed finite field. For fixed integers $1\le a\le b$, Theorem~\ref{thm:detailed} and the uniform-minor reduction in Section~\ref{sec:introduction} give the same three families, with the projective and affine geometries over fields of prime-power order $q<2\max\{2,a,b-a\}$. Since there are only finitely many such field orders, their argument also applies to matroids with no $U_{a,b}$-minor. Let $\cF_d(M)$ denote the set of rank-$d$ flats of $M$.

\begin{corollary}\label{cor:ramsey}
For all integers $1\le a\le b$, $c\ge1$, and $1\le d\le k$, there is an integer $R$ such that every $c$-coloring of $\cF_d(M)$, where $M$ is simple with no $U_{a,b}$-minor and $r(M)\ge R$, has a rank-$k$ flat $F$ for which all members of $\cF_d(M|F)$ have the same color.
\end{corollary}

\subsection{Local density}

The three families in Theorem~\ref{thm:unavoidable} also give a density interpretation.

\begin{corollary}\label{cor:density}
For all integers $1\le a\le b$ and $n\ge2$, every sufficiently high-rank simple matroid $M$ with no $U_{a,b}$-minor and no rank-$n$ independent flat has a rank-$n$ flat $F$ such that $|F|\ge2^{n-1}$.
\end{corollary}

By Theorem~\ref{thm:unavoidable}, $M$ has a rank-$n$ flat $F$. Since $M$ has no rank-$n$ independent flat, $M|F$ is a projective or affine geometry over $\GF(q)$ for some prime power $q$. Both projective and affine geometries of rank $n$ over $\GF(q)$ have at least $q^{n-1}\ge2^{n-1}$ points.

If $M$ is representable over a fixed finite field $\mathbb{F}$ of characteristic $p$, then Theorem~\ref{thm:finite-field} gives a rank-$n$ flat $F$ with $|F|\ge p^{n-1}$.

\subsection{Matroids of girth at least five}

Geelen~\cite[Conjecture~15]{Geelen21} conjectured that every simple matroid of sufficiently large rank and girth at least five has a large independent flat. Projective geometries of rank at least two contain triangles, as do
affine geometries of rank at least two over fields of order at least three. Binary affine geometries are triangle-free, but those of rank at least three contain four-element circuits. Finally, $\AG(1,2)$ is free. Thus, by Theorem~\ref{thm:unavoidable}, for every positive integer $n$, a sufficiently high-rank simple matroid with no $U_{a,b}$-minor and girth at least five has a rank-$n$ flat whose restriction is free, and therefore has a rank-$n$ independent flat. A similar conclusion for matroids representable over a prime field with girth at least five appeared in~\cite[Corollary~1.3]{GK24}.

\begin{corollary}\label{cor:girth}
For all integers $1\le a\le b$ and $n\ge1$, every sufficiently high-rank simple matroid $M$ with no $U_{a,b}$-minor and girth at least five has a rank-$n$ independent flat.
\end{corollary}

\medskip
\noindent\textbf{AI declaration.} In an initial draft of this manuscript, the author attempted to extend Theorem~\ref{thm:finite-field} to classes of matroids excluding a long-line minor. For that problem, the author used a result of Nelson and Norin~\cite[Theorem~1]{NN21} together with the Growth Rate Theorem~\cite{GKW09} to force a large projective geometry minor. When the author was stuck on passing from a projective geometry minor to a representable flat, OpenAI's ChatGPT (6.1 Sol) found a proof of an initial version of Lemma~\ref{lem:pg-to-flat}. In particular, it pointed out the relevance of~\cite[Lemma~7.2]{GN17}. The author then realized that a similar conclusion might hold for matroids excluding a fixed uniform matroid as a minor (the main result of the present paper), if we can bypass the usage of the Growth Rate Theorem. ChatGPT then brought~\cite[Theorem~1.1]{GN15} to the author's attention and assisted with the proof of Lemma~\ref{lem:exponential-covering}.

The author also used ChatGPT for proofreading, language editing, and organizing references. The author wrote the final versions of all arguments, has checked the content, and takes full responsibility for it.

\end{document}